\documentclass{article}
\usepackage[utf8]{inputenc}
\usepackage[english]{babel}
\usepackage[T1]{fontenc}
\usepackage{amsmath}
\usepackage{amsfonts}
\usepackage{amssymb}
\usepackage{amsthm}
\usepackage{hyperref}
\usepackage{lineno}
\usepackage{color}

\newtheorem{theorem}{Theorem}

\newtheorem{lemma}{Lemma}[section]
\newtheorem{coro}[lemma]{Corollary}
\newtheorem{propo}[lemma]{Proposition}

\title{Superidentities for the algebras $UT_2$ and  $UT_3$ \\on a finite field}
\author{Ronald Ismael Quispe Urure \thanks{This study was financed by the Coordenação de Aperfeiçoamento de Pessoal de Nível Superior - Brasil (CAPES) - Finance Code 001.}
	\\
	Departamento de Matemática, Universidade Federal de Juiz de Fora\\
	36036-900, Juiz de Fora - MG, Brazil\\
	e-mail: \texttt{urure6@gmail.com}\\ 
	\\
	Tatiana Aparecida Gouveia 
	\\
	Departamento de Matemática, Universidade Federal de Juiz de Fora\\
	36036-900, Juiz de Fora - MG, Brazil\\
	e-mail: \texttt{tatiana.gouveia@ufjf.edu.br}}

\begin{document}
	
\maketitle



\begin{abstract}
Let $ F $ be a finite field and consider $ UT_n $ the algebra of $ n\times n $ upper triangular matrices over $ F $.
In \cite{valentizaicev2007}, it was proved that every $ G $-grading is elementary.
In \cite{divincenzo2004}, the authors classified all nonisomorphic elementary $ G $-gradings.
They also described the set of all $G $-graded polynomial identities
for $ UTn $ when $ F $ is an infinite field.
In \cite{riva2020}, was described the all $ G $-graded polynomial identities for $ UT_n $ when $ F $ is a finite field.
In this work, we will discuss   the  case when $G = \mathbb{Z}_2$, $ n=2, 3 $ and $ F $ is a finite field. 
\end{abstract}


\section{Preliminaries}	

\hspace{0.4cm} 
Let $\mathcal{A}$ be an associative algebra over a field $F$.
The (left normed) \textit{commutator} of elements in $\mathcal{A}$ is defined inductively by   
$ 	[u_{1},u_{2}] = u_{1} u_{2} - u_{2}u_{1}$ and 
$	[u_{1},	\ldots ,u_m	] = [[u_{1},	\ldots,	u_{m-1}], u_{m} ] $ for $m>2$, $ u_{1},	\ldots ,u_m \in \mathcal{A} $.
We will need of a notation for the powers inside of the commutators. Define inductively, $[u, v^{(0)}] = u$ and $[u,v^{(r)}] = [u,v^{(r-1)},v] $ for $r\geqslant 1$, $u,v \in \mathcal{A}$.

We say that  $\mathcal{A}$ is a \textit{superalgebra} with grading $(\mathcal{A}^{(0)}, \mathcal{A}^{(1)})$
if $\mathcal{A} = \mathcal{A}^{(0)} \oplus \mathcal{A}^{(1)}$ where $\mathcal{A}^{(0)}$ and $\mathcal{A}^{(1)}$ are vector subspaces
satisfying $\mathcal{A}^{(0)}\mathcal{A}^{(0)} + \mathcal{A}^{(1)}\mathcal{A}^{(1)} \subseteq \mathcal{A}^{(0)}$
and $\mathcal{A}^{(0)}\mathcal{A}^{(1)} + \mathcal{A}^{(1)} \mathcal{A}^{(0)} \subseteq \mathcal{A}^{(1)}$.
Let  $\mathcal{B}$ a subalgebra of $\mathcal{A}$. 
If $ \mathcal{B} $ is a superalgebra with grading $(\mathcal{A}^{(0)}\cap \mathcal{B}, \mathcal{A}^{(1)}\cap \mathcal{B})$,
then we say that $\mathcal{B}$ is a \textit{graded subalgebra} of $\mathcal{A}$ induced  by grading $ (\mathcal{A}^{(0)}, \mathcal{A}^{(1)}) $.

Let $\mathcal{A}$ and $\mathcal{B}$ be superalgebras and consider $\varphi: \mathcal{A} \rightarrow \mathcal{B}$ a homomorphism of algebras. We say that $\varphi$ is a \textit{homomorphism of superalgebras} if 
$ \varphi $ preserves the grading, i.e.,
$\varphi (\mathcal{A}^{(0)}) \subseteq \mathcal{B}^{(0)}$
and $\varphi (\mathcal{A}^{(1)}) \subseteq \mathcal{B}^{(1)}$.
The endomorphisms and  isomorphisms  of superalgebras  are  defined in the usual way.

Let  $F\langle Y \cup Z \rangle$ be the free (non unitary) associative algebra over $F$, freely generated 
by $ Y\cup Z $ where  
$Y	=\{y_1,y_2,\ldots \}$ and $Z=	\{z_1,z_2,\ldots \}$.
The elements of  
$  F\langle Y\cup Z \rangle$ are polynomials in non-commutative  variables in $ Y\cup Z $ with scalars in $F$.	
We say that $ p=[x_1,\ldots, x_m] $ is a \textit{commutator of variables in} $ Y\cup Z $ if $ x_i \in Y\cup Z $ for all $i=1,\ldots,m$. 
The number of indices $i$ such that $x_i$ lies in $Z$ is said to be  the \textit{degree of $p$ in $Z$}, e.g., the degree of $ [z_1,y_1,z_1,z_2] $ in $Z$ is $3$.  
It is known that each polynomial  in
$ F\langle Y\cup Z \rangle$  can be written as linear combination of
\begin{equation}\label{basis} 
	y_1^{r_1}	\cdots	y_{m}^{r_m}
	z_{1}^{s_1}	\cdots	z_{n}^{s_n}	p_1 \cdots p_l 
\end{equation}
where  
$ r_i, s_i $ are nonnegative integers, $p_i$ are commutators of variables in $Y\cup Z$, $l \geqslant 0$
(we will use the convention that $ p_1 \cdots p_l = 1 $ if $l=0$),
$ m,n\in \mathbb{N} $
and  $\displaystyle\sum_{i=1}^m r_i + \displaystyle\sum_{i=1}^n s_i + l \geqslant 1$.

Let
$\mathcal{F}^{(1)}$  be the vector subespace spanned by all monomials of $ F\langle Y\cup Z \rangle $ in which the variables in $Z$ occur in odd number and let
$\mathcal{F}^{(0)}$ be the vector subespace spanned by the remaining monomials of $ F\langle Y\cup Z \rangle $. It is well-known that the free algebra $ F\langle Y\cup Z \rangle $ is a superalgebra with the grading $ (\mathcal{F}^{(0)} , \mathcal{F}^{(1)}) $.
One immediately see that the subset of $ \mathcal{F}^{(0)} $ consisting by polynomials in which not appear variables in $Z$ is the associative free algebra generated by $Y$, which we denote by $ F\langle Y \rangle $.

We say that  $f=f(y_1,\ldots,y_m, z_1, \ldots, z_n) \in F\langle Y\cup Z \rangle$ is a \textit{polynomial superidentity} for the superalgebra
$\mathcal{A} = \mathcal{A}^{(0)} \oplus \mathcal{A}^{(1)}$ if 
\[f(u_1, \ldots, u_m, v_1,\ldots, v_n)=0\]
for all $u_1, \ldots, u_m \in \mathcal{A}^{(0)}$ and $v_1,\ldots,v_n \in \mathcal{A}^{(1)}$. 
By $ \textrm{Id}^{gr} (\mathcal{A}) $ we denote 
the set of all  polynomial superidentities of  $\mathcal{A} = \mathcal{A}^{(0)} \oplus \mathcal{A}^{(1)}$, which has structure of $ T_2$-ideal, that is,
is an ideal invariant by endomorphisms of the superalgebra  $ F\langle Y\cup Z \rangle $. 
Alternatively, we also use the notation 
$\textrm{Id}(\mathcal{A}^{(0)} , \mathcal{A}^{(1)})$ for $\textrm{Id}^{gr} (\mathcal{A})$, if we want specify the grading.

Let $ I $ be a $ T_2 $-ideal of $ F\langle Y\cup Z \rangle $.
A subset $S$ of $F\langle Y\cup Z \rangle$ is called a \textit{basis} of  $I$ if $S$ generates $I$ as a $T_2$-ideal, we denote by $I = \langle S \rangle _{T_2}$.

The following Lemma is a generalization of identity $ [uv,w] = u[v,w] + [u,w] v $.

\begin{lemma}\label{lemma uv}
	Let $\mathcal{A}$ be an associative algebra over $F$. Then,
	\begin{align*}
		[uv,w^{(r)}] = \sum_{i=0}^{r}		\binom{r}{i}	[u , w^{(i)}]	[ v,w^{(r-i)}]	
	\end{align*}
	for all $u,v,w \in \mathcal{A}$, $r\geqslant 0$.
\end{lemma}

\begin{proof}
	We will  show the case $r=2$. The general case can be done by induction.
	By using the case $r=1$ sometimes, we see that 
	\[[uv,w^{(2)}] = [u[v,w]	, w] + [[u,w] v , w ]	=	u [v,w^{(2)}] + 2[u,w] [ v, w ] + [u,w^{(2)} ] v.	\]	
\end{proof}

\begin{lemma}\label{non z's}
	$z_1 f z_2 \in \langle z_1z_2	\rangle_{T_2}$ for all $f \in F\langle Y\cup Z \rangle$. 
\end{lemma}

\begin{proof}
	Write $f= f^{(0)} + f^{(1)}$ where $f^{(i)} \in \mathcal{F}^{(i)}$, $i=0,1$. Since, $z_1f^{(0)} \in \mathcal{F}^{(1)}$ we have that 
	$z_1f^{(0)}z_2 \in \langle	z_1z_2 \rangle_{T_2} $. On the other hand
	$z_1 f^{(1)} \in 	\langle	z_1z_2 \rangle_{T_2}$,
	thus, $z_1fz_2	=	z_1	f^{(0)} z_2 +	z_1 f^{(1)}	z_2 \in \langle	z_1z_2 \rangle_{T_2} $.
\end{proof}

\begin{lemma}\cite[Theorem 5.2.1]{drenskybook}\label{lema comutador ordenado}
	Let $p$ be a commutator of variables in $Y \cup Z $. 	
	Then $p = v + v'$ where $v$ is a linear combination of ordered commutators of variables in $Y \cup Z $ and
	$v'$ is a linear combination of products of at least two ordered commutators of variables in $Y \cup Z $. 
\end{lemma}

From now on, we only will consider a finite field $ F $ of $q$ elements.
It is known that  $a^{q-1} =1$ for all $a \in F- \{0\}$. Therefore, $a^{q} =a$ holds for all $a \in F$.	

\begin{lemma}\label{field identities}
	\cite[Lemma 2.1]{ronalddimas}\cite[Proposition 4.2.3]{drenskybook} 
	The monomials
	$
	x_{i_1}^{r_1}\cdots x_{i_m}^{r_m}$, $i_1 <\ldots<i_m$,  $1\leqslant r_1,\ldots, r_m < q	$, $ m\in \mathbb{N}$,
	are linearly independent modulo $Id(F)$, where $Id(F)$ is the $T$-ideal of ordinary polynomial identities of $F$.
\end{lemma}

Consider the following order on the variables $Y \cup Z$: 
\[ y_1 < y_2 < \ldots < y_n < \ldots <z_1 < z_2 < \ldots < z_m < \ldots .\]
We say that a commutator $ [x_{i_1},\ldots,x_{i_m}] $  of variables in $Y\cup Z$ is \textit{ordered} if $x_{i_1} > x_{i_2} \leqslant x_{i_3} \leqslant \ldots \leqslant x_{i_m}$.
For example, the multilinear ordered commutators in the variables $z_1, y_1,y_3,y_4$ are 
$ [z_1,  y_1,y_3,y_4 ], [y_3,  y_1,y_4,z_1 ],  [y_4,  y_1,y_3,z_1 ] $. 

We say that a polynomial $p$ in variables in $Y$ is an \emph{ordered $q$-commutator} if  either  
$  p=	[y_{j_1},y_{j_2}^{(s_2)},y_{j_1}^{(s_1-1)}, y_{j_3}^{(s_3)} , \ldots, y_{j_{n}}^{(s_{n})}] $ 
where the indices $j_i$ are distinct with
$j_1> j_2 < \ldots < j_{n}$,  
$1 \leqslant s_1,  \ldots, s_{n} < q$, $n \geqslant 2$, or
$$p =	
[y_{l_1}^q - y_{l_1}, y_{l_2}^{(t_2)} ,  \ldots , y_{l_{k}}^{(t_{k})}]$$
where the indices $l_i$ are distinct with
$ l_2 < \ldots < l_{k}$,  
$0 \leqslant t_2,  \ldots, t_{k} < q$, $k \in \mathbb{N}$.

\begin{lemma}\cite[Lemma 2.5]{siderov}\label{siderov}
	Let $\mathcal{A}$ be an algebra associative over 
	$F$. 
	Then,
	$[u, v ^{(q)}] = [u, v ^ q]$
	for all $u,v \in \mathcal{A}$.
\end{lemma}
\begin{proof}
	One can verify that 
	\[uv^r = \sum_{i=0}^{r}		\binom{r}{i}	v^{i}	[u,v^{(r-i)}],	\]
	for all $u,v \in \mathcal{A}$.
	Taking $r=q$ and using the fact that 
	$ \binom{q}{m} = 0 \mod{(\textrm{char} \, F)} $, for all  $0 < m < q$, we obtain the assertion of the Lemma.
\end{proof}

Throughout this paper we must mentioned sometimes, the following sets of $ F \langle Y \rangle $:  
$ \Omega_1=\{ [y_1 , y_2], y_2^q - y_2 \}$, $ \Omega_2=\{ [y_3 , y_4], y_4^q - y_4 \}$. We will use the  short notation $ \Omega_i=\{ [y_{2i-1} , y_{2i}], y_{2i}^q - y_{2i} \} $, $ i=1,2 $.

\begin{lemma}\label{triangular basis}\cite{siderov}
	A basis for the ordinary polynomial  identities of $UT_2$ is given by $ w_1w_2 $ where $ w_i \in \Omega_i  $. Moreover,
	a linear basis for the vector space 
	$F\langle Y \rangle / \langle w_1w_2 \mid w_i \in \Omega_i \rangle_{T}$  consist by the polynomials of type
	$$y_{1}^{r_1} \ldots y_{m}^{r_m} p^{\theta} \mod \langle w_1w_2 \mid w_i \in \Omega_i \rangle_T,$$
	where  
	$0 \leqslant r_1, \ldots,r_m < q$,
	$p$ is an ordered $q$-commutator,
	$\theta =0, 1$, $\displaystyle \sum_{i=1}^m r_i + \theta \geqslant 1 $.
	
	Let $L(Y)$ be the free Lie algebra generated by $Y$. Then, a basis for the vector space $L(Y) / \langle [w_1 , w_2] \mid w_i \in \Omega_i \rangle_{L_T}	$ is given by the set of all $q$-ordered commutators module $\langle [w_1 , w_2] \mid w_i \in \Omega_i \rangle_{L_T}$. 
\end{lemma}

\section{Superidentities of $UT_2$}

\hspace{0.4cm} 
We will denote by $ e_{ij} $,  $1\leqslant i\leqslant j \leqslant 2$,  {the elementary matrices} of the $UT_2$.  
Consider $UT_2^{(0)} = \textrm{span}_F \{e_{11},e_{22}\}$ and 
$UT_2^{(1)} = \textrm{span}_F \{e_{12}\}$.
The algebra $UT_2$ with the grading $(UT_2^{(0)},UT_2^{(1)})$ is a superalgebra. Such grading is called canonical grading of $UT_2$.

\begin{propo}\label{propo super ident}
	The following are polynomial superidentities for $UT_2$:
	\begin{equation}\label{superidentities}
		z_1z_2, \qquad [y_1,y_2], \qquad y_1^q-y_1.
	\end{equation}
\end{propo}

\begin{proof}
	We prove only that $y_1^q-y_1 \in \textrm{Id}^{gr}(UT_2)$.
	Let $A=a e_{11}+ b e_{22}$ be an element of $UT_2^{(0)}$. Since $|F|=q$, we have that $A^q=a^q e_{11}+ b^q e_{22}= A$.
\end{proof}

Let $I = \langle z_1z_2, [y_1,y_2], y_1^q-y_1	\rangle_{T_2}$.

\begin{propo}\label{generators}
	The vector space $F\langle Y \cup Z \rangle / I$ is generated module $I$ by   polynomials of the type
	\begin{align}\label{equation generators}
		y_{1}^{r_1}	\cdots	y_{m}^{r_m}	[z_j, y_{j_1}^{(s_1)}, \ldots,	y_{j_n}^{(s_n)}]^{\theta}
	\end{align}
	where  $j_1 < \ldots < j_{n}$,  $0 \leqslant r_1, \ldots, r_m ,
	s_1,\ldots, s_n < q$, $ \theta \in \{0,1\} $, $ \displaystyle\sum_{i=1}^m r_i + \theta \geqslant 1 $, $ m,n,j \in \mathbb{N}$.
\end{propo}

\begin{proof} 
	We claim that $p_1p_2 \in I$ for all $p_1,p_2$ commutators of variables in $Y \cup Z$.
	To verify this, observe that $[z_1,z_2] \in I$ and
	$[y_1,y_2] \in I,$ {thus} we have that $I$ contains all the commutators in $\mathcal{F}^{(0)}$. 
	On the other hand,   
	$[y_1,z_1] [y_2,z_2] \in \langle	z_1z_2 \rangle_{T_2} \subseteq I$ because $[y_1,z_1]$ and $ [y_2,z_2]$ are in $\mathcal{F}^{(1)}$. 
	Therefore, 
	by (\ref{basis}),
	the vector space		$F\langle Y \cup Z \rangle / I$ is generated by elements of type  
	\begin{equation*} 
		y_{1}^{r_1}	\cdots	y_{m}^{r_m},	\qquad
		y_{1}^{r_1}	\cdots	y_{m}^{r_m}z_j,	\qquad
		y_{1}^{r_1}	\cdots	y_{m}^{r_m}	p,
	\end{equation*}
	where $p$ is an commutator of variables in $Y \cup Z$ such that $p\in \mathcal{F}^{(1)}$. Moreover, since $y_1^q  = y_1\mod I$, we can suppose that the powers $r_i$ are less that $q$.

	Now, since $z_1z_2 \in I$, remains consider the case when the degree of $p$  in $Z$ is $ 1 $.
	As $[y_1,y_2]\in I$ 
	we have that all commutators of variables in $Y$ vanished in  $F\langle Y \cup Z \rangle / I$, then we can suppose that $p  = \pm [z_j,y_{i_1}	,	\ldots, y_{i_k}	] \mod I$ for some $j$ and indices $i_t$.
	Moreover,
	we have commutativity of variables in $Y$ inside of $p$, which is possible  by  
	Jacobi's identity:
	\[ [f,y_{l},y_k] = [f,y_{k},y_l] + [y_{k},y_l, f] = [f,y_{k},y_l] \mod I, \]
	for all
	$ f \in F\langle Y \cup Z \rangle$.
	Thus, we can suppose that 
	$p =  [z_j,y_{j_1}^{(s_1)}	,	\ldots, y_{j_n}^{(s_n)}	]	\mod	I$, where $j_1< \ldots < j_n$ and $s_i$ are nonnegative integers.
	By Lemma 
	\ref{siderov} we know that $[z_1,y_1^{(q)}]  = [z_1,y_1^q]$ and since $y_1^q-y_1 \in I$ we have that $[z_1,y_1^{(q)}]  = [z_1,y_1] \mod I$. Therefore, we can suppose that  $s_i <q$ for all $i=1,\ldots , n$. 
\end{proof}

\begin{theorem}\label{linear independence}
	Let $F$ be a finite field of $q$ elements.
	A basis for the superidentities of $UT_2$ with the canonical grading is given by set (\ref{superidentities}). Moreover, a linear basis for the vector space $F\langle Y\cup Z \rangle / \emph{Id}^{gr}(UT_2)$ consist {of} the set (\ref{equation generators}).
\end{theorem}

\begin{proof}
	We begin by first showing that the set (\ref{equation generators}) above is linearly independent module $\textrm{Id}^{gr}(UT_2)$. Let $f = g + h \in \textrm{Id}^{gr}(UT_2)$ where  
	$g$ and $h$ be a linear combination of polynomial of type (\ref{equation generators})  with $ \theta =0 $ and $ \theta = 1, $ respectively. 
	Since neither of the variables {in} $Z$ occur in $g,$ we can write $ f(y_1,\ldots, y_m,z_1,\ldots,z_n) = g(y_1,\ldots, y_m) + h(y_1,\ldots, y_m,z_1,\ldots,z_n)$ for some $ m,n\in \mathbb{N} $. Making  $z_j=0$  
	we have that $h(y_1,\ldots, y_m,0,\ldots,0) = 0$. Therefore,  $g \in Id^{gr}(UT_2)$ and then $h\in Id^{gr}(UT_2)$ too.

	Let $a_1, \ldots, a_m $ be arbitrary elements in $F$.
	If we evaluate $y_i$ {by} $a_ie_{11} \in UT_2^{(0)}$, we obtain $g(a_1e_{11},\ldots,a_m e_{11}) = g(a_1,\ldots,a_m)e_{11}$. Then, since 
	$g \in \textrm{Id}^{gr}(UT_2) $
	we have that $g$ is an ordinary polynomial	identity for the field $F$. By Lemma \ref{field identities} we have that $g=0$. 
	Now,	we can assume that  
	\[h(y_1,\ldots,y_m, z_1 )=\sum_{0\leq s_i,r_j <q} \alpha_{\mathbf{s},\mathbf{r}} y_1^{r_1}	\cdots y_m^{r_m}
	[z_1, y_1^{(s_1)},	\ldots, y_m^{(s_m)} ]\]
	where
	$\alpha_{\mathbf{s},\mathbf{r}} \in F$ and $\mathbf{s}= \{s_i\}_{i=1}^m$, $\mathbf{r}= \{r_i\}_{i=1}^m$ are $m$-uples.
	Define $A_i= a_i e_{11} + (a_i+b_i) e_{22} \in UT_2^{(0)}$ where  $ a_i,b_i \in F, i=1,\ldots m$. 
	Observe that $[e_{12},A_i] = b_i e_{12}$ and so $[e_{12},A_i,A_j]  = b_i b_j e_{12}$. In general,
	$[e_{12}, A_1^{(s_1)},	\ldots, A_m^{(s_m)} ] = b_1^{s_1}\cdots b_m^{s_m}	e_{12}$. 
	Thus, if we evaluate $y_i$ by
	$A_i $ and $z_1$ by $ e_{12} $ we obtain that
	\begin{equation*}
		h(A_1,\ldots,A_m, e_{12} )= \sum_{0\leqslant s_i,r_j <q} \alpha_{\mathbf{s},\mathbf{r}} a_1^{r_1}\cdots	a_m^{r_m}	b_1^{s_1}	\cdots b_m^{s_m}	e_{12}	=	0.
	\end{equation*}
	Since $ a_i,b_i $ are arbitrary, we have that all coefficients $\alpha_{\mathbf{s,r}}$ are zero, by Lemma \ref{field identities}. Thus, $h=0$.

	By Proposition \ref{propo super ident} we have that $I\subseteq \textrm{Id}^{gr}(UT_2)$. 
	For the converse inclusion,
	let
	$f\in \textrm{Id}^{gr}(UT_2)$. By Proposition \ref{generators} there exists  $f'$ a linear combination of (\ref{equation generators}) 
	such that $f=f' \mod I$. 
	Hence, we have that $ f = f' \mod \textrm{Id}^{gr}(UT_2) $	since $I\subseteq \textrm{Id}^{gr}(UT_2)$. Then, by the previous paragraph $f'=0$ and so $f \in I$.
\end{proof}

\section*{Superidentities for $UT_3$}

\hspace{0.4cm} 
In the same manner as the previous section,
we will denote by $ e_{ij} ,	1\leqslant i\leqslant j \leqslant 3$, {the elementary matrices} of $UT_3$.
In the following two sections we will obtain basis of superidentities for $UT_3$ 
with the following nontrivial elementary gradings 
$	(\mathcal{A}^{(0)}, \mathcal{A}^{(1)})$ and $(\mathcal{B}^{(0)},\mathcal{B}^{(1)})$, where
\begin{align*}
	&\mathcal{A}^{(0)} = \textrm{span}_F \{ e_{11},e_{22} , e_{33}, e_{23} \}  & &\mathcal{B}^{(0)} = \textrm{span}_F\{ e_{11},e_{22} , e_{33}, e_{13} \}\\
	&\mathcal{A}^{(1)} = \textrm{span}_F\{ e_{12} , e_{13}\}   & &\mathcal{B}^{(1)} = \textrm{span}_F\{ e_{12},e_{23} \}.
\end{align*}

\section{Superidentities of $UT_3$ with the grading $(\mathcal{A}^{(0)},\mathcal{A}^{(1)})$}

\begin{propo}\label{propo superidentities A}
	The following are polynomial superidentities for $UT_3$ with the grading $ (\mathcal{A}^{(0)} , \mathcal{A}^{(1)})$:
	\begin{align}\label{superidentities of A}
		z_1 z_2, \qquad w_1 z_1, \qquad w_1 w_2,
	\end{align}
	where
	$w_i \in \Omega_i $, $ i=1,2 $.
\end{propo}
\begin{proof}
	It	is easy to see that $z_1z_2 \in \textrm{Id}(\mathcal{A}^{(0)}, \mathcal{A}^{(1)})$.
	We may complete the proof by proving that the evaluation of $[y_1,y_2]$ and  $ y_1^q-y_1 $ on $ \mathcal{A}^{(0)} $ are a scalar multiples of $e_{23}$.
	We only evaluate the polynomial 
	$y_1^q-y_1$. 
	Let ${A}=a e_{11} + b e_{22} + c e_{33} + d e_{23} \in \mathcal{A}^{(0)} $. 
	One can be show that 
	$ 	{A}^q = a e_{11} + b e_{22} + c e_{33} + dd_0 e_{23} $,
	where $ d_0 = b^{q-1}+b^{q-2}c + \ldots + c^{q-1} $. 
	Therefore,
	${A}^q - A= d (d_0 -1)e_{23}$.
\end{proof}

In the proof the Proposition \ref{propo superidentities A}, we claim that $ d_0 = (c-b)^{q-1}$. Indeed, observe that 
$ d_0 (c-b)  = c^q-b^q = c-b $.
Thus, if $ c\neq b $ then $ d_0=1 $.
Otherwise if $ c=b $ then $ d_0 =qb^{q-1}  = 0 $. 
Therefore,
\begin{equation}\label{equation q commutator}
	{A}^q - A= d ((c-b)^{q-1} -1)e_{23}.
\end{equation} 
We will need this equation in the proof of Theorem \ref{linear independence of A} below.

Now let us consider $M=\langle w_1z_1, w_1w_2	\mid	w_i \in \Omega_i,i=1,2	\rangle_{T_2}$.

\begin{lemma}\label{lemma one variable}
	Let $p$ be a commutator of variables in $Y\cup Z$. If $p$ has  degree $1$ in $Z$, then $p$ can be write in the form
	\[ p= \alpha[ z_j, y_{j_1}^{(s_1)}, \ldots, y_{j_n}^{(s_n)} ] +z_j h + g \mod  M,\]
	where	$ \alpha \in F $,
	$j_1< \ldots <j_n $, $ s_i$ are nonnegative integers, $ n,j \in \mathbb{N} $,
	$h$ is a linear combination of commutators of variables in $Y$ and $g$ is a linear combination of products of at least two ordered commutators of variables in $Y \cup Z $.  
\end{lemma}

\begin{proof}
	Let 
	$p$ be a commutator with one only variable $z_j$ of degree $1$ and with at least two variables in $Y$.
	By Lemma \ref{lema comutador ordenado}, we can write $p$ as
	\[p= \alpha [z_j,y_{j_1}^{(s_1)},\ldots , y_{j_n}^{(s_n)}] + \sum_{{\delta}} \beta_{\delta} [p_{\delta}, z_j] + g, \]
	for some $\alpha,\beta_{\delta} \in F$, $j_1 < \ldots < j_n$, 
	$\delta$ run over some finite set of indices such that 
	$p_{\delta}$
	is a commutator  of variables in $Y$ 
	and $g$ is linear combination of product of at least two ordered commutators of variables in $Y \cup Z$.
	For each $\delta$ as above,  
	we have   $[p_{\delta} , z_i] = - z_i p_{\delta} \mod M$. Thus, $\displaystyle\sum_{{\delta}} \beta_{\delta} [p_{\delta}, z_i] = z_i h \mod M,$ where $h= -\displaystyle\sum_{{\delta}} \beta_{\delta} p_{\delta}$.
\end{proof}

\begin{lemma}\label{lemma z variable q commutator}
	For all $r_1,\ldots,r_n \geqslant 1$ there exist $s_1,\ldots,s_n <q$ such that
	\[  [ z_j, y_{1}^{(r_1)}, \ldots, y_{{n}}^{(r_{n})} ] =  [ z_j, y_{1}^{(s_1)}, \ldots, y_{{n}}^{(s_{n})} ] + \sum_{\delta} \alpha_{\delta} f_{\delta} p_{\delta} \mod M, 
	\]
	where
	$\alpha_{\delta} \in F$ and $\delta$ run over a finite set of indices such that $f_{\delta}$ is of type
	$[ z_j, y_{1}^{(t_1)}, \ldots, y_{n}^{(t_{n})} ]$
	where $0 \leqslant t_1,\ldots,t_n<q$
	and $p_{\delta}$ is an ordered $q$-commutator. 
\end{lemma}

\begin{proof}
	We will proceed by induction on $n$.
	We claim that for $r\geqslant 1$ there exists $1\leqslant s < q$ and $\beta_i \in F$ such that
	\begin{align}\label{equation m=1}
		[z,y^{(r)}] =  [z,y^{(s)}] + \sum_{i=0}^{q-1} \beta_i [z,y^{(i)}](y^q-y) \mod M.
	\end{align}
	Observe that, by Lemma \ref{siderov} we have that $[z,y^{(q)}] = [z,y] + [ z , y^q-y] = [z,y] +  z  (y^q-y) \mod M$. Thus, for every $t$
	\[	[z,y^{(q+t)}]=[z,y^{(q)},y^{(t)}]= [z,y^{(1+t)}] + [ z ( y^q-y), y^{( t)}].\]
	On the other hand, by using induction on $t$, one can verify that 
	$$[ z ( y^q-y), y^{( t)}] = [ z , y^{( t)}] ( y^q-y) \mod M$$ and since $( y^q-y)^2 \in M$ we can repeat the process for $[z,y^{(1+t)}]$ and $[ z , y^{( t)}],$ therefore we obtain the equation (\ref{equation m=1}).
	
	Suppose that $n>1$ and that 
	\[ p= [ z_j, y_{1}^{(r_1)}, \ldots, y_{{n-1}}^{(r_{n-1})} ] =  [ z_j, y_{1}^{(s_1)}, \ldots, y_{{n-1}}^{(s_{n-1})} ] + \sum_{\delta} \alpha_{\delta} f_{\delta} p_{\delta} \mod M,\]
	where
	$\alpha_{\delta} \in F$ and $\delta$ run over a finite set of indices such that $f_{\delta}$ is a ordered commutator of the type
	$ 	[ z_j, y_{1}^{(t_1)}, \ldots, y_{n}^{(t_{n})} ] $,
	$p_{\delta}$ is an ordered $q-$commutator and the powers $s_i,t_i$ satisfy the conditions of Lemma.   
	Using the fact that $p\in \mathcal{F}^{(1)}$ we can  substitute $z$ by $p$  in (\ref{equation m=1}). That is,  
	\begin{align*}
		[p, y_{n}^{(r_n)}] = [p, y_{n}^{(s_n)}] + \sum_{0 \leqslant i<q} \beta_i  [ p ,y_{n}^{(i)}](y_{n}^q-y_{n})
		\mod M.
	\end{align*}
	By our induction assumption we have
	$$[p, y_{n}^{(i)}] 
	=[ z_j, y_{1}^{(s_1)}, \ldots, y_{{n-1}}^{(s_{n-1})} , y_{n}^{(i)} ] +
	\sum_{\delta} \beta_{\delta} [f_{\delta} p_{\delta} , y_{n}^{(i)} ] 
	\mod M, $$
	for all $i=0, \ldots, q-1$.
	Thus, remains to show that for every $ \delta $ and every $i$ the polynomials 
	$ [f_{\delta} p_{\delta} , y_{n}^{(i)} ] $ and $ [f_{\delta} p_{\delta} , y_{n}^{(i)} ] (y_{n}^q-y_{n}) $ have the desired form.
	Indeed, 
	by Lemma 
	\ref{lemma uv} we have
	\[
	[f_{\delta} p_{\delta} , y_{n}^{(i)} ] = 
	\sum_{0\leqslant l \leqslant i} \binom{i}{l}[f_{\delta} , y_{n}^{(l)}  ] [ p_{\delta} , y_{n}^{(i-l)} ].
	\] 
	Note that each of $[ p_{\delta} , y_{n}^{(i-l)} ]$
	and 
	$ y_{n}^q-y_{n} $ are ordered $q-$commutators, $[f_{\delta} , y_{n}^{(l)}  ]$ is in the required form and 
	$ [f_{\delta} p_{\delta} , y_{n}^{(i)} ] (y_{n}^q-y_{n}) = 0 \mod M $, for all $\delta, i, l$ as above. 
\end{proof}

Denote by $J$ the $T_2$-ideal generated by the set (\ref{superidentities of A}). We have that $ M \subseteq J $.

\begin{lemma}\label{lemma ppp}
	Let 
	$p_1,p_2$ be two commutators of variables in $ Y\cup Z $.
	If  $p_1p_2\notin J$ then $p_1$ has degree $ 1 $ in $Z$ and $p_2$ has no variable in $Z$.   Consequently,
	$p_1p_2p_3 \in J$.
\end{lemma}

\begin{proof}
	First, observe that by Lemma \ref{non z's}, the total sum of degree of $p_1$ and $ p_2$ in $ Z $ is no more than $2$.
	Suppose that $p_1,p_2$ have no variable in $Z,$ 
	soon $p_1p_2 \in \langle [y_1,y_2] [y_3,y_4]\rangle_{T_2}.$ 
	If each of $p_1$ and $p_2$ has degree $ 1 $  in $Z$, 
	it follows that $p_1p_2 \in \langle z_1z_2\rangle_{T_2}.$ 
	Finally, when $p_1$ has no variable in $Z$ and $p_2$ has degree $ 1 $ in $Z$, we obtain that $p_1p_2 \in \langle [y_1,y_2]z_1 \rangle_{T_2}.$ 
	
	The last affirmation of the Lemma follows immediately. 
\end{proof}

\begin{propo}
	The vector space $F \langle Y\cup Z\rangle / J $ is generated module $J$ by
	polynomials of type
	\begin{align}\label{generators module J}
		y_{1}^{r_1}	\cdots	y_{{m}}^{r_{m}} [z_j,y_{j_1}^{(s_1)}, \ldots, y_{{j_n}}^{(s_{n})}]^{\theta_1} 
		p^{\theta_2} 
	\end{align}
	where  $ j_1< \ldots < j_n $, $0\leqslant r_1,\ldots,r_{m} , s_1,\ldots,s_{n}< q$,  $ m,n,j \in \mathbb{N} $,
	$p$ is an ordered $q$-commutator,	
	$\theta_1,\theta_2 \in \{0,1\}$ and 
	$ \displaystyle\sum_{i=1}^m r_i +\theta_1+\theta_2 \geqslant 1 $.
\end{propo}

\begin{proof}
	By (\ref{basis}) and by Lemma \ref{lemma ppp} the vector space $F \langle Y\cup Z\rangle / J $ is generated by polynomials of the type
	$y_{1}^{l_1}	\cdots	y_{{m}}^{l_{m}} p_1^{\theta_1} 
	p_2^{\theta_2} \mod J$ where, $p_1\in Z$ or $p_1$ is a commutator of variables in $Y\cup Z$ with degree $1$ in $Z$, $p_2$ is a commutator of variables in $Y$ and $\theta_1,\theta_2 \in \{0,1\} $. 
	By Lemma \ref{lemma one variable}  {it suffices consider} $ p_1 $ in the form  $ 
	[z_{j},y_{j_1}^{(s_1)} , \ldots,y_{j_{n}}^{(s_n)}] $, $ j_1< \ldots < j_n $. 
	Better yet,	   we can suppose yet that 
	$0 \leqslant s_1,\ldots,s_{n} <q $ by Lemma \ref{lemma z variable q commutator}, however, we must added the ordered $ q $-commutators as possibilities for $p_2$.
	By Lemma \ref{triangular basis}	  we can suppose that all $ p_2 $ are ordered $q$-commutators.
	For this purpose, we will use that \[ \langle [w_1 , w_2] \mid w_i \in \Omega_i \rangle_{L_T} \subseteq 
	\langle [w_1 , w_2] \mid w_i \in \Omega_i \rangle_{T_2} \subseteq J . \] 
	
	Finally, again by Lemma \ref{triangular basis}, we have that 
	$ y_{1}^{l_1}	\cdots	y_{{m}}^{l_{m}} $ is a linear combination of	$ y_{1}^{r_1}	\cdots	y_{{m}}^{r_{m}}	p^{\theta} $ where $ 0 \leqslant r_1,\ldots,r_m < q$, $ p $ is an ordered $ q $-commutator and $ \theta =0,1 $. Since that $ p p_1 p_2 \in J,$ 
	it follows that we can only consider $\theta=0$, 
	which concludes this proof.
\end{proof}

\begin{theorem}\label{linear independence of A}
	Let $F$ be a finite field of $q$ elements.
	A basis for the superidentities of $UT_3$ with the grading $(\mathcal{A}^{(0)},\mathcal{A}^{(1)})$ is given by set (\ref{superidentities of A}). Moreover, a basis for the vector space $F\langle Y\cup Z \rangle / \emph{Id}(\mathcal{A}^{(0)},\mathcal{A}^{(1)})$ consist of the set given in (\ref{generators module J}).
\end{theorem}

\begin{proof}
	Let us first show that the set given in (\ref{generators module J}) is linearly independent module 
	$\textrm{Id}(\mathcal{A}^{(0)},\mathcal{A}^{(1)})$. 
	Let $f$ be a linear combination of (\ref{generators module J}) such that $f\in \textrm{Id}(\mathcal{A}^{(0)},\mathcal{A}^{(1)})$. We can write $f = f_0 + f_1$ 
	where 
	$ f_0 $ and $ f_1 $ are the sumands of $f$ which $\theta_1 = 0$ and $ \theta_1 = 1, $ respectively.
	Since neither variable in $ Z $ occur in $ f_0 $ we have that $ f_0 \in  \textrm{Id}(\mathcal{A}^{(0)},\mathcal{A}^{(1)})$ 
	and so
	$ f_1 \in  \textrm{Id}(\mathcal{A}^{(0)},\mathcal{A}^{(1)})$ too.
	
	Note that  
	\begin{equation}\label{equation f_0}
		f_0 \in  F \langle Y \rangle \cap \textrm{Id}(\mathcal{A}^{(0)},\mathcal{A}^{(1)}) \subseteq \textrm{Id}(\mathcal{A}^{(0)}) \subseteq  F \langle Y \rangle,
	\end{equation}
	where $ \textrm{Id}(\mathcal{A}^{(0)}) $ denotes the $ T $-ideal of ordinary polynomial identities of the algebra $ \mathcal{A}^{(0)} $. 
	How the algebra $\mathcal{A}^{(0)}$ contains a copy of $ UT_2 $ we have that $\textrm{Id} (\mathcal{A}^{(0)})  \subseteq \textrm{Id}(UT_2)$ 
	and then
	$ f_0=0 $ by Lemma \ref{triangular basis}.

	Now, we can write $f_1 = g_0 + g_1$ 
	where
	$ g_0 $ and $ g_1 $ are the summands of $ f_1 $ such that $\theta_2=0$ and $\theta_2=1$, respectively.
	One immediately see that $ g_1 \in \textrm{Id}^{gr}(UT_2)$ since each 
	$p \in \textrm{Id}^{gr}(UT_2) $. 
	Note that if $ f_1 \in \textrm{Id}^{gr}(UT_2) $ then $ g_0 \in \textrm{Id}^{gr}(UT_2) $ and thus have $ g_0 = 0 $ by Theorem \ref{linear independence}.
	We will show so that $ f_1\in  \textrm{Id}^{gr}(UT_2)$.
	For this we consider $UT_2$ as a graded subalgebra of $ \mathcal{A}^{(0)} \oplus \mathcal{A}^{(1)} $. 
	More specifically,  
	$ \mathcal{C}^{(0)} = \textrm{span}_F \{e_{22}, e_{33}\}$ and $ \mathcal{C}^{(1)} = \textrm{span}_F \{e_{13}\} $ are subspaces of $ \mathcal{A}^{(0)} $ and $ \mathcal{A}^{(1)} $ respectively. One can verify that 
	$ \mathcal{C}^{(0)} \oplus \mathcal{C}^{(1)} $ is a graded subalgebra of $ \mathcal{A}^{(0)} \oplus \mathcal{A}^{(1)} $ isomorphic to 
	$ UT_2 $ with the canonical graduation.
	Hence, $	\textrm{Id}(\mathcal{A}^{(0)} , \mathcal{A}^{(1)} )	\subseteq \textrm{Id}(\mathcal{C}^{(0)} , \mathcal{C}^{(1)})  = 	\textrm{Id}^{gr}(UT_2) 	  $ and therefore, $ f_1 \in \textrm{Id}^{gr}(UT_2)  $.
	
	Thus, we can suppose that $ f= f_1(y_1,\ldots, y_m, z_1)$ is a linear combination of 
	$  y_{1}^{r_1}	\cdots	y_{{m}}^{r_{m}} [z_1,y_{1}^{(s_1)}, \ldots, y_{{m}}^{(s_{m})}] p $
	with $ r_i,s_i $ and $ p $ as in (\ref{generators module J}).
	Let us write $ f = h+h_q $ where $ h $ and $ h_q $ are the summands of $f$ such that $p$ is of type  
	$ [y_{j_1},y_{j_2}^{(t_2  )}, y_{j_1}^{(t_1-1)} , y_{j_3}^{(t_3)} , \ldots, y_{j_{n}}^{(t_{n})}] $ 
	and 
	$ [y_{l_1}^q - y_{l_1}, y_{l_2}^{(t_2)} ,  \ldots , y_{l_{k}}^{(t_{k})}] $, 
	respectively.
	Define
	$A_i=a_ie_{11}+(a_i+b_i)e_{22}+(a_i+b_i+c_i)e_{33}+d_ie_{23} \in \mathcal{A}^{(0)} $
	where $ a_i,b_i,c_i,d_i \in F,  i=1,\ldots,m $. 
	Straightforward calculations gives
	\begin{align*}
		A_1^{r_m}	\cdots A_m^{r_m} [e_{12}, A_1^{(s_1)},	\ldots, A_m^{(s_m)} ]  =	a_1^{r_1}	\cdots a_m^{r_m} b_1^{s_1}\cdots	b_m^{s_m}e_{12}+\beta e_{13}
	\end{align*}
	for some
	$ \beta \in F$, and
	\begin{align*}
		[A_{j_1}, A_{j_2}^{(t_2)}, A_{j_1}^{(t_1-1)}, A_{j_3}^{(t_3)}	,	\ldots, A_{j_n}^{(t_{n})} ]
		=	(d_{j_1}c_{j_2}-d_{j_2}	c_{j_1}) c_{j_1}^{t_1-1}	c_{j_2}^{t_2-1}	c_{j_3}^{t_3} \cdots c_{j_n}^{t_{n}}	e_{23}.
	\end{align*}
	By using the equation (\ref{equation q commutator}) one can verify that
	\begin{align*}
		[A_{l_1}^q-A_{l_1}, A_{l_2}^{(t_2)}	\ldots, A_{l_k}^{(t_{k})} ] 		=	
		d_{l_1}	(c_{l_1}^{q-1}-1)	c_{l_2}^{t_2}\cdots c_{l_k}^{t_{k}}e_{23}.
	\end{align*}
	
	Suppose that $ h\neq 0 $. There exists some term 
	\[	y_{1}^{r_1}	\cdots	y_{{m}}^{r_{m}} [z_1,y_{1}^{(s_1)}, \ldots, y_{{m}}^{(s_{m})}]
	[y_{j_1},y_{j_2}^{(t_2)}, y_{j_1}^{(t_1-1)}, y_{j_3}^{(t_3)} , \ldots, y_{j_{n}}^{(t_{n})}]
	\]
	of $h$ 
	such that their coefficient $ \alpha_{\mathbf{r,s,j,t}} $,
	where 
	$\mathbf{r}=\{r_i\}_{i=1}^m$,
	$\mathbf{s}=\{s_i\}_{i=1}^m$
	are $ m $-uples, $\mathbf{j}=\{j_i\}_{i=1}^n$ and $\mathbf{t}=\{t_i\}_{i=1}^n$ are $n$-uples, is nonzero.	Let $ j_1^* $ be  the maximum between all the indices $ j_1 $ such that  $ \alpha_{\mathbf{r,s,j,t}} \neq 0 $.
	Making $d_{j_1^*} = c_{j_1^*}$ and $ d_{i} = 0 $ for all $ i\neq j_1^* $, we proceed to evaluate  $ y_i $ by $ A_i $ and $ z_1 $ by $e_{12}$ in $f$.
	Since $ c_{j_1^*} (c_{j_1^*}^{q-1}-1) = 0 $ we have that  $ h_q(A_1,\ldots,A_m,e_{12})=0 $ and 
	\begin{align*}
		h	(A_1,\ldots,A_m,e_{12})	=	\sum_{\mathbf{s,r,\hat{j},t}}	\alpha_{\mathbf{s,r,j,t}}	a_1^{r_1}	\cdots a_m^{r_m} b_1^{s_1}\cdots	b_m^{s_m}	 c_{j_1^*}^{t_1}	c_{j_2}^{t_2}	c_{j_3}^{t_3} \cdots c_{j_n}^{t_{n}} e_{23}=0
	\end{align*}
	where $ \mathbf{\hat{j}} = \{j_i\}_{i=2}^n $.
	Note that  the maximality of $ j_1^* $ guarantees that the possible  term of $ h $ with $j_2=j_1^*$ not appear. 
	Thus, since $ a_i,b_i,c_i $ are arbitrary, by Lemma \ref{field identities} follows that $ \alpha_{\mathbf{r,s,{j},t}}=0 $ for all $\mathbf{r,s},\hat{\mathbf{j}},\mathbf{t} $, a contradiction. 
	Therefore, $ h=0 $.
	
	An analogous situation occur with $ h_q $. Suppose that there exists a term 
	\[ y_{1}^{r_1}	\cdots	y_{{m}}^{r_{m}} [z_1,y_{1}^{(s_1)}, \ldots, y_{{m}}^{(s_{m})}] [y_{l_1}^q - y_{l_1}, y_{l_2}^{(t_2)} ,  \ldots , y_{l_{k}}^{(t_{k})}]  \]
	of $h_q$ 
	such that their coefficient $ \beta_{\mathbf{r,s,l,t}} $,
	is nonzero.
	Let $ l_1^* $ be  the maximum integer between all the indices $ l_1 $ such that  $ \beta_{\mathbf{r,s,l,t}} \neq 0 $.
	We evaluate $ y_i $ by $ A_i $ and $ z_1 $ by $e_{12}$ in $ h_q $ with the additional condition 
	$d_{l_1^*} = - 1$, $ c_{l_1^*} = 0 $ and $ d_{i} = 0 $ for all $ i\neq l_1^* $.
	Then 
	\begin{align*}
		h_q	(A_1,\ldots,A_m,e_{12})	=	\sum_{\mathbf{s,r,\hat{l},t}}	\beta_{\mathbf{s,r,l,t}}	a_1^{r_1}	\cdots a_m^{r_m} b_1^{s_1}\cdots	b_m^{s_m}	 	 		c_{l_2}^{t_2}\cdots c_{l_k}^{t_{k}}	e_{23}=0
	\end{align*}
	where $ \mathbf{\hat{l}} = \{l_i\}_{i=2}^k $. Again by Lemma \ref{field identities}, we arrive at a contradiction. 
	
	The rest of the proof can be done
	in the same way as in the final paragraph of the proof of the Theorem \ref{linear independence}.  
\end{proof}

\section{Superidentities of $UT_3$ with the grading $( \mathcal{B}^{(0)}, \mathcal{B}^{(1)})$}
Recall that 	$ \mathcal{B}^{(0)} = \textrm{span}_F\{ e_{11},e_{22} , e_{33}, e_{13} \}$ and   $
\mathcal{B}^{(1)} = \textrm{span}_F\{ e_{12},e_{23} \}$.

\begin{propo} 
	The following are polynomial superidentities for $UT_3$ with the grading $(\mathcal{B}^{(0)} , \mathcal{B}^{(1)})$:
	\begin{align}\label{superidentities of B}
		z_1 z_2 z_3, \qquad z_1 w_1, \qquad w_1 z_1, \qquad w_1 w_2,
	\end{align}
	where
	$w_i \in \Omega_i$, $ i=1,2 $.
\end{propo}

\begin{proof}
	We begin with the following observation, if
	$C \in \mathcal{B}^{(0)}$ is a matrix with zeros on the diagonal, then $CD = DC = 0$ for all $D \in \mathcal{B}^{(1)}$. 
	Thus, since  the evaluations of $z_1 z_2, [y_1,y_2]$ and $y_1^q-y_1 $ have zeros on the diagonal, 
	$ \textrm{Id}(\mathcal{B}^{(0)} , \mathcal{B}^{(1)}) $
	contains the first three polynomials from (\ref{superidentities of B}). 
	On the other hand, note that the product of two matrices in $ \mathcal{B}^{(0)} $, with zeros on the diagonal, is zero. 
	Then, $w_1w_2 \in \textrm{Id}(\mathcal{B}^{(0)}, \mathcal{B}^{(1)})$.
\end{proof}

Consider $N=\langle w_1 z_1 ,z_1w_1, w_1  w_2 \mid w_i \in \Omega_i \rangle_{T_2}$  with $\Omega_i=\{ [y_{2i-1} , y_{2i}], y_{2i}^q - y_{2i} \},i=1,2$.

\begin{lemma}\label{lemma one variable B}
	
	If
	$p$ is a commutator of variables in $Y\cup Z$ of degree $ 1 $ in  $Z$, 
	then, $p = \pm [ z_j, y_{1}^{(s_1)}, \ldots, y_{{m}}^{(s_{m})} ]  + g\mod N$, 
	where  $0 \leqslant	s_1,\ldots,s_m <q$,	 $ g $ is a linear combination of products of at least two ordered commutators of variables in $ Y\cup Z $ and
	$ j,m \in \mathbb{N}$.
\end{lemma}

\begin{proof}Follows of the Lemma \ref{lemma one variable} and Lemma \ref{lemma z variable q commutator}, since $M \subseteq N$.
\end{proof}

\begin{coro}\label{coro one variable B}
	If 
	$p$ is a commutator of variables in $Y\cup Z$ of degree $ 2 $ in  $Z$,  
	then, $ p = \pm [ z_j, y_{1}^{(s_1)}, \ldots, y_{{m}}^{(s_{m})} , z_k	]  + g\mod N $,
	where $0 \leqslant	s_1,\ldots,s_m <q$,
	$ g $ is a polynomial as in Lemma \ref{lemma one variable B} and $ j,k,m \in \mathbb{N} $.
\end{coro}

\begin{proof}
	By Lemma \ref{lema comutador ordenado}, $ p $ has the form $ p = \pm [ z_j, y_{1}^{(t_1)}, \ldots, y_{{m}}^{(t_{m})} , z_k	]  + g\mod N $ where $ g $ is as in Lemma \ref{lemma one variable B} and $ t_i $ are integers. By  Lemma \ref{lemma one variable B} there exists $ 0 \leqslant s_1, \ldots,s_m <q $ such that 
	\begin{align*}
		p = \pm [ z_j, y_{1}^{(s_1)}, \ldots, y_{{m}}^{(s_{m})} , z_k	]  + [g',z_k] + g\mod N
	\end{align*} 
	where $ g' $ has the same property as $ g $.
	Moreover,
	by using the identity $ [x_1 x_2, x_3] = x_1[x_2,x_3] + [x_1 x_3] x_2 $ it is easy to verify that $ [g',z_k] $ is a linear combination of products of at least two commutators of variables in $ Y \cup Z $.
\end{proof}

Denote by $Q$ the $T_2$-ideal generated by the set (\ref{superidentities of B}).   We have that $ N \subseteq	Q $.

\begin{lemma}\label{lemma ppQ}
	Let 
	$p_1,p_2$ be two commutators of variables in $ Y\cup Z $.
	Suppose that 
	the product $p_1p_2 \notin Q$. Then $p_1$ and $p_2$ have degree $ 1 $ in $Z$.
	Moreover, if $p_3$ is another commutator of variables in $Y\cup Z $ then $p_1p_2p_3 \in Q$. 
\end{lemma}

\begin{proof}
	Since $z_1z_2z_3 \in Q$ we have that the total sum of degree of $p_1$ and $p_2$ in  $Z$  is no more that $2$.
	
	Suppose that $p_1$ have no variable in $Z$, (e.g.  $p_1=[y_{3},y_4,y_1]$).
	Since $[y_1 , y_2]z_1 \in Q$ we have that $p_2 \in \mathcal{F}^{(0)}$. Write $p_2=[u,x_{t}] = u x_t - x_t u$.
	If $x_{t} \in Z$ we have that $u\in \mathcal{F}^{(1)}$, therefore,  $ p_1 u , p_1 x_{t} \in \langle [y_1 , y_2]z_1 \rangle_{T_2}$ and so
	$p_1p_2  \in \langle [y_1 , y_2]z_1 \rangle_{T_2} $,
	which is absurd. Otherwise, if $x_{t} \in Y$ we have that 
	$u\in \mathcal{F}^{(0)}$ and
	$p_1p_2=p_1[u,x_{t}] \in \langle [y_1,y_2] [y_3,y_4]\rangle_{T_2}$, which is absurd too. Thus, the degree of $ p_1 $ in $Z$ is $ \geqslant 1$.	
	Analogously we can obtain that  the degree of  $p_2$ in $Z$ is $ \geqslant 1$. Hence, $p_1$ and $p_2$ have degree $ 1 $ in $Z$.
	
	Now, let $p_3$ be another commutator of variables in $Y\cup Z $. If $p_3$ have no variable in $Z$ then
	$ p_2p_3 \in Q$ by previous paragraph.
	Otherwise, if $p_3$ has degree $\geqslant 1$ then, $p_1p_2p_3\in Q$.
\end{proof}

\begin{propo}
	The vector space $F \langle Y\cup Z\rangle / Q $ is generated module $Q$ by
	\begin{align}\label{generators module Q}
		a)\ 
		y_{1}^{r_1}	\cdots	y_{m}^{r_m} p^{\theta} \qquad
		b) \  y_{1}^{r_1}	\cdots	y_{m}^{r_m}	[z_{j},y_{1}^{(s_1)} , \ldots,y_{m}^{(s_{m})}]	
		[z_{k},y_1^{(t_1)} , \ldots,y_{m}^{(t_{m})}]^{\mu}
	\end{align}  
	where $0 \leqslant r_i,s_i,t_i < q$, $p$ is an ordered $q-$commutator, $\theta, \mu \in \{0,1\}$, $ \displaystyle\sum_{i=1}^{m} r_i +\theta \geqslant	1 $ and $ m,j,k \in \mathbb{N} $.  
	
\end{propo}

\begin{proof}
	By (\ref{basis}) and by Lemma \ref{lemma ppQ} the vector space $F \langle Y\cup Z\rangle / Q $ is generated module $Q$ by polynomials of the type $y_{1}^{r_1}	\cdots	y_{m}^{r_m} \rho$ where $ \rho $ is one  of the following types of polynomials:
	\begin{align}\label{set propo B}
		p^{\theta}, 	\qquad p_2,	\qquad	z_i p_1^{\theta},	\qquad	z_i z_j,	\qquad	p_1,		\qquad	p_1 p_1',		
	\end{align}
	where
	$p$ a commutator of variables in $Y$, $ \theta = 0 ,1 $,	$ i \leqslant j $,
	$p_1$ and  $ p_1'$ are commutators of variables in $Y\cup Z$ of degree $1$ in $ Z $ and $p_2$ is a commutator of variables in $Y\cup Z$ of degree  $ 2 $ in  $Z $. Moreover, we can suppose that $ r_1,\ldots,r_m <q  $ and that $ p $ is a ordered $ q $-commutator, by Lemma \ref{triangular basis}.

	By Corollary \ref{coro one variable B},  $ p_2 $
	has the form
	\[ [z_{j},y_{1}^{(s_1)} , \ldots,y_{m}^{(s_{m})} ]	z_k	-	z_k [z_{j},y_{1}^{(s_1)} , \ldots,y_{m}^{(s_{m})}]  + g \mod Q, \]
	where $ 0 \leqslant s_1,\ldots,s_m <q $ and $ g $ is a linear combination of two commutators. 
	By Lemma \ref{lemma ppQ} , we can suppose that $ g $ is a linear combination of polynomials of the type $ p_1 p_1' $. 
	Moreover, by Lemma \ref{lemma one variable B}
	it is suffices to consider, instead of 
	$ p_1 $, commutators of type
	$[z_{j},y_{1}^{(s_1)} , \ldots,y_{m}^{(s_{m})}] $, where $0 \leqslant s_1,\ldots,s_m<q $. 
	Thus, by these considerations  the set (\ref{set propo B}) can be rewritten as:
	\begin{align*}
		p^{\theta}, 	\qquad	z_i p_1^{\theta},	\qquad	z_i z_j,		\qquad p_1 z_i,	\qquad	z_j z_i,	\qquad	p_1,	\qquad	p_1 p_1'.		
	\end{align*}
	Finally, by Lemma \ref{lemma one variable B} again, we can also suppose that $ p_1' $ is in the form $[z_{k},y_{1}^{(t_1)} , \ldots,y_{m}^{(t_{m})}] $ where  $0\leqslant t_i <q $, as desired.
\end{proof}

\begin{theorem}
	Let $F$ be a finite field of $q$ elements.
	A basis for the superidentities of $UT_3$ with the grading $(\mathcal{B}^{(0)},\mathcal{B}^{(1)})$ is given by set (\ref{superidentities of B}). Moreover, a linear basis for the vector space $F\langle Y\cup Z \rangle / \emph{Id}(\mathcal{B}^{(0)},\mathcal{B}^{(1)})$ consist of the set (\ref{generators module Q}).
\end{theorem}

\begin{proof}
	We will only
	to show that the set (\ref{generators module Q}) is linearly independent module $\textrm{Id}(\mathcal{B}^{(0)},\mathcal{B}^{(1)})$. 
	Let $f$ be a linear combination of polynomials in (\ref{generators module Q}) such that $f\in \textrm{Id}(\mathcal{B}^{(0)},\mathcal{B}^{(1)})$.
	Write $ f=f_y + f_z $ where $ f_y $ and $ f_z $ are the summands  of $ f $ which the terms are of the type (\ref{generators module Q}.$ a) $) and (\ref{generators module Q}.$ b) $), respectively.
	Since neither variable in $ Z $ occur in $ f_y $ we have that $ f_y \in  \textrm{Id}(\mathcal{B}^{(0)},\mathcal{B}^{(1)})$ and then $ f_z \in  \textrm{Id}(\mathcal{B}^{(0)},\mathcal{B}^{(1)})$.

	Note that (\ref{equation f_0}) is yet valid for $ f_y $ and $ \mathcal{B} $, inside of $ f_0 $ and $ \mathcal{A} $, respectively.
	The algebra $ \mathcal{B}^{(0)} $ contains a copy of 
	$ UT_2 $. 
	Thus $ f_y=0 $ 
	by Lemma \ref{triangular basis}.

	Let us write $f_z = f_0 + f_1$ 
	where
	$ f_0 $ and $ f_1 $ are the summands of $ f_z $ such that their terms satisfy $\mu=0$ and $\mu=1$, respectively.
	Observe that $ f_1 \in \textrm{Id}^{gr}(UT_2)$ 
	since  each  product
	$ 	[z_{j},y_{1}^{(s_1)} , \ldots,y_{m}^{(s_{m})}]	
	[z_{k},y_1^{(t_1)} , \ldots,y_{m}^{(t_{m})}] \in \textrm{Id}^{gr}(UT_2)$.
	
	We claim that $ f_z\in  \textrm{Id}^{gr}(UT_2)$.
	For this, we will consider $UT_2$ as a graded subalgebra of $ \mathcal{B}^{(0)} \oplus \mathcal{B}^{(1)} $. 
	More specifically,  
	set
	$ \mathcal{D}^{(0)} = \textrm{span}_F \{e_{11}, e_{22}\}$ and $ \mathcal{D}^{(1)} = \textrm{span}_F \{e_{12}\} $ linear subspaces of $UT_3$.
	One can verify that 
	$ \mathcal{D}^{(0)} \oplus \mathcal{D}^{(1)} $ is a graded subalgebra of $ \mathcal{B}^{(0)} \oplus \mathcal{B}^{(1)} $ isomorphic to 
	$ UT_2 $ with the canonical graduation.
	Hence, $	\textrm{Id}(\mathcal{B}^{(0)} , \mathcal{B}^{(1)} )	\subseteq \textrm{Id}(\mathcal{D}^{(0)} , \mathcal{D}^{(1)})  = 	\textrm{Id}^{gr}(UT_2) 	  $.  Therefore $ f_z \in \textrm{Id}^{gr}(UT_2)  $. This is implies  $ f_0 \in \textrm{Id}^{gr}(UT_2) $ and soon
	$f_0 = 0 $ by Theorem \ref{linear independence}.
	
	Now, let us write $ f_1 = \displaystyle \sum_{jk} f_{jk}$ where the $ f_{jk} $ is obtained of $ f_1 $ fixing $ z_j $ and $ z_k $ in 
	$$ y_{1}^{r_1}	\cdots	y_{m}^{r_m}	[z_{j},y_{1}^{(s_1)} , \ldots,y_{m}^{(s_{m})}]	
	[z_{k},y_1^{(t_1)} , \ldots,y_{m}^{(t_{m})}] .$$ 
	Without loss of generality we may assume that $ j,k \in \{1,2\} $, for this, we just substitute $ 0 $ for the remaining variables in $ Z $.  
	That is, we may assume that 
	$ f_1= f_1(y_1,\ldots,y_m,z_1,z_2) =  f_{11}+	f_{12}+	f_{21}+f_{22}$, since we can rename the variables when necessary.

	For making
	$ z_2 =0$  we have  $ f_1(y_1,\ldots,y_m,z_1,0) =  f_{11} \in \textrm{Id}(\mathcal{B}^{(0)} , \mathcal{B}^{(1)} )$.
	Take 
	$A_i=a_ie_{11}+ (a_i+b_i) e_{22}+(a_i+b_i+c_i)e_{33} \in \mathcal{B}^{(0)}$ 
	where $ a_i,b_i,c_i \in F , i=1,\ldots m$. 
	One can verify that  
	$[e_{12}, A_1^{(s_1)},	\ldots, A_m^{(s_m)} ] = b_1^{s_1}\cdots b_m^{s_m} e_{12}$ and 
	$[e_{23}	, A_1^{(s_1)},	\ldots, A_m^{(s_m)} ] = c_1^{s_1}\cdots c_m^{s_m}e_{23} $.
	Therefore,
	\begin{align*}
		\ A_1^{r_1}\cdots A_m^{r_m} 	[e_{12}+e_{23}	, A_1^{(s_1)},	\ldots, A_m^{(s_m)} ]&	[e_{12}+e_{23}, A_1^{(t_1)},	\ldots, A_m^{(t_m)} ]	=	\nonumber  \\
		=&	a_1^{r_1}\cdots a_m^{r_m}	b_1^{s_1}\cdots b_m^{s_m} 	c_1^{t_1}\cdots c_m^{t_m}	e_{13}.
	\end{align*} 
	If
	$ 
	\displaystyle{f_{11} = \sum_{ r_i, s_i,t_i =0	}^{q-1}	\alpha_{\mathbf{r,s,t}}	y_{1}^{r_1}	\cdots	y_{m}^{r_m}	[z_1,y_{1}^{(s_1)} , \ldots,y_{m}^{(s_{m})}]	
		[z_{1},y_1^{(t_1)} , \ldots,y_{m}^{(t_{m})}]}
	$
	then 
	\begin{align*}
		f_{11}	(A_1,\ldots,A_k,e_{12}+e_{23}) = \sum_{ r_i, s_i,t_i =0}^{q-1}	\alpha_{\mathbf{r},\mathbf{s},\mathbf{t}}	
		a_1^{r_1}\cdots a_m^{r_m}	b_1^{s_1}\cdots b_m^{s_m} 	c_1^{t_1}\cdots c_m^{t_m}	e_{13},
	\end{align*} 
	where $ \alpha_{\mathbf{r,s,t}} \in F $ and $ \textbf{r} = {\{r_i\}}_{i=1}^m,\mathbf{s} = {\{s_i\}}_{i=1}^m ,	\mathbf{t}= {\{t_i\}}_{i=1}^m $ are $ m $-uples. 
	Hence, since  $ f_{11}	(A_1,\ldots,A_k,e_{12}+e_{23})	= 0 $ and $ a_i,b_i,c_i $ are arbitrary, the coefficients $ \alpha_{r,s,t} $ are equals to zero by 
	Lemma \ref{field identities}. 
	
	Analogously, by making $ z_1=0 $ and $ z_2=e_{12} + e_{23}$, one can to show that the coefficients of $ f_{22}  $ are all equals to zero. 
	Furthermore,  the evaluation  $ z_1$ by $e_{12} $, $ z_2$ by ${e_{23}} $  solves the problem for  $ f_{12} $, and the evaluation $ z_1$ by $e_{23} $, $ z_2$ by ${e_{12}} $ serves for $ f_{21} $.

\end{proof}

\bibliographystyle{tfnlm}
 
\bibliography{artigoversaofinalbibtex}

\end{document}